\documentclass[a4paper,11pt]{article}
\usepackage{hyperref}
\hypersetup{
    colorlinks=true,
    linktoc=all,
    linkcolor=red,     
    citecolor=blue,     
}
\usepackage{graphicx}
\usepackage{amssymb}
\usepackage{latexsym, bm}
\usepackage{multicol}
\usepackage{indentfirst}
\usepackage{amsfonts}
\usepackage{amsmath}
\usepackage{setspace}
\newtheorem{theorem}{Theorem}[section]
\newtheorem{lemma}{Lemma}[section]
\newtheorem{claim}{Claim}

\newtheorem{problem}{Problem}[section]

\newcommand{\ignore}[1]{}

\begin{document}
\begin{spacing}{1}

\title{Ramsey numbers of $K_s+mK_t$ versus $K_n$}
\date{}

\author{
Lulu Dai\footnote{Center for Discrete Mathematics, Fuzhou University,
Fuzhou, 350108, P.~R.~China. Email: {\tt 1415088965@qq.com}. } \;\;  and \;\; Qizhong Lin\footnote{Center for Discrete Mathematics, Fuzhou University,
Fuzhou, 350108, P.~R.~China. Email: {\tt linqizhong@fzu.edu.cn}. Supported in part  by National Key R\&D Program of China (Grant No. 2023YFA1010202) and NSFC (No.\ 12571361).}
}

\maketitle
\begin{abstract}
For integers $m \ge 1$, $s \ge 0$, and $t \ge 1$, let $K_s + mK_t$ denote the join graph of a clique $K_s$ and $m$ vertex-disjoint copies of $K_t$. 
We prove that for fixed integers $m\ge 1$, $t\ge 1$ and $s\ge 0$, 
\[
  R(K_s+mK_t,K_n)=O\!\left(\frac{n^{s+t-1}}{(\log n)^{s+t-2}}\right).
\] 
This settles a problem proposed by Liu and Li (2026).
This bound is tight up to a constant factor for the case $(s,t)=(0,3)$, as it matches the classical result $R(K_3, K_n) = \Theta\!\left(n^2 / \log n\right)$ of Kim (1995).
\end{abstract}

\section{Introduction}

For graphs $G$ and $H$, the Ramsey number $R(G,H)$ is the minimum $N$ such that every red-blue coloring of the edges
of the complete graph $K_N$ contains a red copy of $G$ or a blue copy of $H$.
Equivalently, $R(G,K_n)$ is the minimum $N$ such that every graph on $N$ vertices that is $G$-free
contains an independent set of size $n$.
Estimating the Ramsey number $R(K_s,K_n)$ is a fundamental and notoriously difficult problem in combinatorics, which has attracted the most attention since 1935.


Fix integers $m\ge 1$, $s\ge 0$ and $t\ge 1$ and define  $K_s+mK_t$ as the join graph of a clique $K_s$ and $m$ vertex-disjoint copies of $K_t$. This graph generalizes several classical structures, such as complete graphs, book graphs, and fan graphs.
Let $B_m^{(s)}=K_s+mK_1$ and $F_m=K_1+mK_2$ denote the book graph and fan graph, respectively. In particular, Rousseau and Sheehan \cite{R-S} showed that $r(B_m^{(2)},B_m^{(2)})=4m+2$ if $4m+1$ is a prime power.
A breakthrough result of Conlon \cite{Conlon} established that for each fixed $s\ge2$ and sufficiently large $m$,
\begin{align*}
r(B_m^{(s)},B_m^{(s)})=2^sm+o(m).
\end{align*}
The error term $o(m)$ of the upper bound has been improved to $O\big(\frac{m}{(\log\log\log m)^{1/25}}\big)$ by Conlon, Fox and Wigderson \cite{c-f} using a different method. We refer the reader to \cite{cl,cly,fly,nr} and references therein. 

The Ramsey numbers of fans also attract considerable attention. In particular, Lin, Li and Dong \cite{lld} showed that for any fixed $m\ge1$ and sufficiently large $n$,  $$r(F_m,F_n)=4n+1.$$
Zhang, Broersma and Chen \cite{T4} proved that the above equality holds if $n\ge \max\{m^2- \frac{m}{2},\frac{11m-8}{2}\}$. Moreover, they showed that for $m\le n\le \frac{m^2-m}{2}$,
\begin{align}\label{zhang}
r(F_m,F_n)\ge 4n+2.
\end{align}

Specifically, it is known that
\[
\frac{9n}{2}-5\le r(F_n,F_n)\le \frac{31n}{6}+15,
\]
where the lower bound is due to Chen, Yu and Zhao \cite{cyz} while the upper bound is proved by Dvo\u{r}\'{a}k and Metrebian \cite{T7} recently.
For more references, we refer the reader to \cite{c-l,lr,ll,zc} etc.

In this note, we investigate a natural generalization of these structures, namely the Ramsey number of the general graph $K_s+mK_t$ against a complete graph $K_n$.

Recently, Liu and Li \cite{Liu-Li} proved the following result.
\begin{theorem}[Liu and Li \cite{Liu-Li}]\label{liu-l}
Let $m\ge 1$, $t\ge 1$ and $s\ge 0$ be fixed integers. Then the following holds.

\medskip
(i) $R(K_s+mK_1,K_n)\le(m+o(1))\frac{n^{s}}{(\log n)^{s-1}}.$

\medskip
(ii) $R(K_s+mK_2,K_n)\le(1+o(1))\frac{n^{s+1}}{(\log n)^{s}}.$

\medskip
(iii) If $t\ge3$, then $R(K_s+mK_t,K_n)=O\!\left(\frac{n^{s+t-1}\log\log n}{(\log n)^{s+t-2}}\right).$
\end{theorem}

They asked whether the factor $\log\log n$ in the above theorem can be removed. 

\begin{problem}[Liu and Li \cite{Liu-Li}]\label{prb}
Let $m\ge 1$, $t\ge 3$ and $s\ge 0$ be fixed integers. Whether
\[
  R(K_s+mK_t,K_n)=O\!\left(\frac{n^{s+t-1}}{(\log n)^{s+t-2}}\right).
\]
\end{problem}

We settle this problem in the affirmative.

\begin{theorem}\label{thm:main}
Let $m\ge 1$, $t\ge 1$ and $s\ge 0$ be fixed integers. Then 
\[
  R(K_s+mK_t,K_n)=O\!\left(\frac{n^{s+t-1}}{(\log n)^{s+t-2}}\right).
\]
\end{theorem}

This bound is tight up to a constant factor for the case $(s,t)=(0,3)$, as it matches the classical result $R(K_3, K_n) = \Theta\!\left(n^2 / \log n\right)$ of Kim \cite{kim}. The cases when $t=1$ and $t=2$ have been established by Liu and Li \cite{Liu-Li}.

\section{Useful lemmas}

For a graph $G$, write $V(G)$ and $E(G)$ for its vertex and edge set, $e(G)=|E(G)|$,
$d(v)$ for the degree of a vertex $v$, and $d=d(G)$ for the average degree.
For $v\in V(G)$, let $N(v)$ be the neighborhood of $v$ and let $G_v:=G[N(v)]$ be the
subgraph induced by $N(v)$.

We will use the following triangle-counting lemma by Ajtai, Koml\'os and Szemer\'edi \cite{AKS81} (see \cite{aks} for a more general result).

\begin{lemma} [Ajtai, Koml\'os and Szemer\'edi \cite{AKS81}] \label{triangle}
For every $\eta>0$ there exists a constant $c=c(\eta)>0$ such that the following holds.
Let $G$ be a graph on $N$ vertices with average degree at most  $D$.
If $G$ contains at most $N D^{2-\eta}$ triangles, then
\[
\alpha(G)\ \ge\ c(\eta)\,\frac{N\log D}{D}.
\]
\end{lemma}

For the proof, a function from \cite{Li-R-96} is used.
\[
  f_m(x)=\int_0^1 \frac{(1-u)^{1/m}}{m+(x-m)u}\,du\qquad (x\ge 0).
\]

\begin{lemma}[Li and Rousseau \cite{Li-R-96}]\label{Li-Rouss}
Let $a \ge 0$ be an integer. Let $G$ be a graph with $N$ vertices and average degree $d$. For any vertex $v$ of $G$, if $G_v$ has maximum degree at most $a$, then  
\[
  \alpha(G) \ge Nf_{a+1}(d)\ge N\frac{\log(d/(a+1))-1}{d}.
\]
\end{lemma}

\ignore{We will use the well-known estimate

\begin{lemma} [Ajtai, Koml\'os and Szemer\'edi \cite{AKS80}] \label{ramseyub}
For every fixed integer $t \ge 2$, we have
\[
R(K_r,K_n)=O\!\left(\frac{n^{r-1}}{(\log n)^{r-2}}\right)
\]
\end{lemma}
}

\section{Proof of Theorem \ref{thm:main}}

Let $m\ge 1$, $t\ge 1$ and $s\ge 0$ be fixed integers. We shall prove that there exists a positive constant $C_s$ depending only on $m,t,s$ such that
$$
  R(K_s+mK_t,K_n)\le C_s\frac{n^{s+t-1}}{(\log n)^{s+t-2}}=:N_s.
$$
The proof will proceed by induction on $s\ge 0$.

\subsection[The case s=0]{The case $s=0$}

It is well known \cite{AKS80} that $R(K_t,K_n)=O(\frac{n^{t-1}}{(\log n)^{t-2}})$ for fixed $t\ge1$, so we have
\begin{equation}\label{s=0}
R(mK_t,K_n) \le (m-1)t+R(K_t, K_n)\le C_0\frac{n^{t-1}}{(\log n)^{t-2}}
\end{equation}
for some constant $C_0>0$,
which matches the statement for $s=0$.

\subsection[The case s=1]{The case $s=1$}
We shall prove $R(K_1+mK_t,K_n)\le C_1\frac{n^{t}}{(\log n)^{t-1}}=:N_1$.
The assertion is trivial if $t=1$ since $R(K_{1,m},K_n)=(n-1)m+1$ from the classical result \cite{cha}. Thus we assume $t\ge2$.
Suppose for a contradiction that there exists a $(K_{1}+mK_{t})$-free graph
$G$ on $N_1$ vertices with $\alpha(G)<n$.

\paragraph{Bounding the maximum degree.}
Since $G$ is $(K_{1}+mK_{t})$-free, for every vertex $v$ the subgraph $G_v$ is $mK_t$-free.
Moreover, $\alpha(G_v)\le \alpha(G)<n$.
Therefore, by using (\ref{s=0}), the maximum degree $\Delta$ of $G$ satisfies that
\begin{equation}\label{eq:d-bound-s1}
  \Delta\le R(mK_t,K_n)-1\le C_0\frac{n^{t-1}}{(\log n)^{t-2}} =:D
\end{equation}

\paragraph{Bounding the number of triangles.}
Let $\tau(G)$ denote the number of triangles in $G$.
We shall show that $\tau(G)\le N_1 D^{2-\eta}$ for some fixed $\eta=\eta(t)\in(0,1)$.

Let $ L= R(K_{t-1},K_n) + (m-1)(t-1) $. A similar argument using \cite{AKS80} yields that
\begin{equation}\label{L-D}
L = O\!\left( \frac{n^{t-2}}{(\log n)^{t-3}} \right)+ (m-1)(t-1) \sim D^{\frac{t-2+o(1)}{t-1}} \le D^{1-\eta}
\end{equation}
for some fixed $\eta$ with $0<\eta<\frac{1}{t-1}$.

The following claim can be used to control the number of edges inside $G_v$.

\begin{claim}\label{clm:heavy}
For every $v\in V(G)$, there are at most $m-1$ vertices $u\in N(v)$ with
\[
\deg_{G_v}(u)=|N(u)\cap N(v)|\ \ge\ L.
\]
\end{claim}
{\bf Proof.}
Suppose to the contrary that there exist distinct vertices $u_1,\dots,u_m\in N(v)$ such that
$|N(u_i)\cap N(v)|\ge L= R(K_{t-1},K_n) + (m-1)(t-1)$ for all $i\in[m]$. Since $\alpha(G)<n$, we can greedily find pairwise vertex-disjoint copies of $K_{t-1}$. Now $u_i$ together with the corresponding 
$K_{t-1}$  spans a $K_t$, and these $K_t$'s are vertex-disjoint and all lie in $N(v)$.
These copies of $K_t$, together with the vertex $v$, form a copy of $K_1+mK_t$ in $G$, a contradiction.
\hfill$\Box$

\medskip
For every fixed $v\in V(G)$, by  (\ref{eq:d-bound-s1}), (\ref{L-D}), and Claim~\ref{clm:heavy}, we have
\begin{align}\label{eG}
2e(G_v)=
\sum_{u\in N(v)}\deg_{G_v}(u)
&\le (m-1)\,d(v) + (d(v)-(m-1))L\nonumber
\\&\le (L+m-1)d(v) \nonumber
\\&\le D^{2-\eta}.
\end{align}
Hence summing (\ref{eG}) over all vertices $v\in V(G)$ and using $\tau(G)=\frac13\sum_{v}e(G_v)$, we obtain
\[
\tau(G)\ \le\ \frac13\sum_{v\in V(G)}e(G_v)
\le N_1D^{2-\eta}.
\]

Since $D=C_0\,n^{t-1}/(\log n)^{t-2}$ from (\ref{eq:d-bound-s1}), we have $\log D \ge \frac{t-1}{2}\log n$ for all sufficiently large $n$. By Lemma~\ref{triangle}, we obtain
\[
\alpha(G)\ \ge\ c(\eta)\,\frac{N_1\log D}{D} \ge\ c(\eta)\,\frac{N_1}{D}\cdot \frac{t-1}{2}\log n
\ =\ c(\eta)\,\frac{t-1}{2}\cdot \frac{C_1}{C_0}\,n.
\]
Choose $C_1$ large enough so that
\[
c(\eta)\,\frac{t-1}{2}\cdot \frac{C_1}{C_0}\ \ge\ 1.
\]
Thus, we obtain $\alpha(G)\ge n$, which contradicts the assumption that $\alpha(G)<n$ and completes the proof.

\subsection[Induction step: s>=2]{Induction step: $s\ge 2$}

Now we assume that Theorem~\ref{thm:main} holds for $s-1$ and $s-2$. 
Suppose for a contradiction that there exists a $(K_s+mK_t)$-free graph $G$ on $N_s$ vertices with $\alpha(G)<n$.
Let $d$ be the average degree of $G$.

Since for any vertex $v\in V(G)$, the subgraph $G_v$ is $(K_{s-1}+mK_t)$-free and
$\alpha(G_v)\le \alpha(G)<n$, 
by the induction hypothesis for $s-1$, we have
\[
  d(v) \le R(K_{s-1}+mK_t,K_n)-1\le N_{s-1}.
\]
Hence the average degree
\begin{equation*}
  d \le N_{s-1}.
\end{equation*}

Moreover, since $G$ is $(K_s+mK_t)$-free,  by the induction hypothesis for $s-2$, we have 
\[
  \Delta(G_v)\le R(K_{s-2}+mK_t,K_n)-1 \le N_{s-2}.
\]

Now we apply Lemma \ref{Li-Rouss} with $d \le N_{s-1}$ and $a\le  N_{s-2}$ to obtain that
\begin{align}
\alpha(G)
&\ge N_s \frac{\log\!\bigl(d/(a+1)\bigr)-1}{d} \notag\\
&\ge \frac{C_s}{C_{s-1}} \frac{n}{\log n}\,
   (\log n-\log\log n) \notag\\
&= \left(\frac{C_s}{C_{s-1}}-o(1)\right)n .
\end{align}

Choosing $C_s$ sufficiently large yields $\alpha(G)\ge n$, contradicting $\alpha(G)<n$.
This completes the induction step, thereby proving Theorem \ref{thm:main}. \hfill$\Box$

\end{spacing}

\end{document}